\documentclass[12pt]{article}
\usepackage{amssymb,amsmath,amsthm,mathrsfs}
\usepackage[usenames]{color}
\usepackage{hyperref}
\usepackage[all]{xy}
\usepackage{xypic}
\usepackage{graphicx}
\usepackage{amscd}

\baselineskip 18pt \textwidth 16cm  \theoremstyle{plain}

\newtheorem{theorem}{Theorem}[subsection]
\newtheorem{proposition}[theorem]{Proposition}
\newtheorem{corollary}[theorem]{Corollary}
\newtheorem{lemma}[theorem]{Lemma}
\newtheorem{definition}[theorem]{Definition}
\newtheorem{notation}[theorem]{Notation}

\newtheorem{remark}[theorem]{Remark}

\newtheorem{prop}[theorem]{Proposition}

\newcommand{\Z}{\mathbb Z}

\newcommand{\nir}[1]{{\color{red}{Nir: #1}}}
\newcommand{\RamiAlt}[1]{}
\newcommand{\RamiAbout}[1]{}

\DeclareMathOperator{\GL}{GL}

\DeclareMathOperator{\Hom}{Hom}
\DeclareMathOperator{\End}{End}
\DeclareMathOperator{\Ad}{Ad}

\DeclareMathOperator{\Res}{Res}

\DeclareMathOperator{\trace}{trace}

\DeclareMathOperator{\Ind}{Ind}

\DeclareMathOperator{\Spec}{Spec}

\DeclareMathOperator{\Gal}{Gal}
\DeclareMathOperator{\Frob}{Frob}
\DeclareMathOperator{\Fr}{Fr}
\DeclareMathOperator{\supp}{supp}

\DeclareMathOperator{\Irr}{Irr}

\usepackage[usenames,dvipsnames]{xcolor}

\definecolor{dblue}{RGB}{6,69,173}
\definecolor{lblue}{RGB}{11,0,128}

\newcommand{\colorlinks}{true}

\newcommand{\linkcolor}{lblue}
\newcommand{\citecolor}{green}
\newcommand{\urlcolor}{dblue}

\newcommand{\linkbordercolor}{red}
\newcommand{\citebordercolor}{green}
\newcommand{\urlbordercolor}{cyan}

\usepackage{hyperref}
\hypersetup{colorlinks=\colorlinks,linkbordercolor=\linkbordercolor, linkcolor=\linkcolor, citecolor=\citecolor, urlcolor=\urlcolor,urlbordercolor=\urlbordercolor,citebordercolor=\citebordercolor}%

\newcommand{\hrefHid}[2]{
\hypersetup{urlbordercolor={1 1 1}}%
\hypersetup{urlcolor=black}%
\href{#1}{#2}%
\hypersetup{urlbordercolor=\urlbordercolor}%
\hypersetup{urlcolor=\urlcolor}%
}

\newcommand{\inhref}[2]{\hyperref[#1]{#2}}

\newcommand{\inhrefHid}[2]{%
\hypersetup{linkbordercolor={1 1 1}}%
\hypersetup{linkcolor=black}%
\inhref{#1}{#2}%
\hypersetup{linkbordercolor=\linkbordercolor}%
\hypersetup{linkcolor=\linkcolor}%
}

\newcommand{\defHref}[3]{\newcommand{#1}[1][#3]{\href{#2}{##1}}}
\newcommand{\defInhref}[3]{\newcommand{#1}[1][#3]{\inhref{#2}{##1}}}
\newcommand{\defHrefHid}[3]{\newcommand{#1}[1][#3]{\hrefHid{#2}{##1}}}
\newcommand{\defInhrefHid}[3]{\newcommand{#1}[1][#3]{\inhrefHid{#2}{##1}}}

\newcommand{\defHrefBoth}[3]{%
\expandafter\defHrefHid \csname #3Hid\endcsname {#1}{#2}%
\expandafter\defHref \csname #3Vis\endcsname {#1}{#2}%
}
\newcommand{\defInhrefBoth}[3]{%
  \expandafter\defInhrefHid \csname #3Hid\endcsname {#1}{#2}%
  \expandafter\defInhref \csname #3Vis\endcsname {#1}{#2}%
}

\newcommand{\defHrefBothVis}[3]{%
\defHrefBoth{#1}{#2}{#3}%
\expandafter\defHref \csname #3\endcsname {#1}{#2}%
}
\newcommand{\defInhrefBothVis}[3]{%
  \defInhrefBoth{#1}{#2}{#3}%
  \expandafter\defInhref \csname #3\endcsname {#1}{#2}%
}

\newcommand{\defHrefBothHid}[3]{%
\defHrefBoth{#1}{#2}{#3}%
\expandafter\defHrefHid \csname #3\endcsname {#1}{#2}%
}
\newcommand{\defInhrefBothHid}[3]{%
  \defInhrefBoth{#1}{#2}{#3}%
  \expandafter\defInhrefHid \csname #3\endcsname {#1}{#2}%
}

\baselineskip 18pt \textwidth 16cm  \theoremstyle{plain}

\newtheorem*{theorem*}{Theorem}
\newtheorem*{remark*}{Remark}
\newtheorem*{conjecture*}{Conjecture}

\newtheorem{introtheorem}{Theorem}

\newtheorem{introconjecture}[introtheorem]{Conjecture}

\newcommand{\C}{\mathbb C}

\DeclareMathOperator{\Tr}{Tr}
\DeclareMathOperator{\Rad}{Rad}

\newcommand{\ACh}{\operatorname{ACh}}
\newcommand{\ind}{\operatorname{ind}}
\newcommand{\IC}{\operatorname{IC}}

\newcommand{\Span}{{\operatorname{Span}}}
\renewcommand{\dim}{{\operatorname{dim}}}

\renewcommand{\Hom}{{\operatorname{Hom}}}

\newcommand{\irr}{\operatorname{Irr}}

\newcommand{\bF}{\mathbb{F}}
\usepackage{varioref}

\title{Bounds on multiplicities of spherical spaces over finite fields}
\author{Avraham Aizenbud\thanks{{The Incumbent of Dr. A. Edward Friedmann Career Development Chair in Mathematics}} \ and Nir Avni}
\begin{document}

\maketitle
\begin{abstract}
Let $G$ be a reductive group scheme of type $A$ acting on a spherical scheme $X$. We prove that {there exists a number $C$ such that
 the multiplicity $\dim\Hom(\rho,\C[X(F)])$ is bounded by $C$, for any finite field $F$ and any irreducible representation $\rho$ of $G(F)$.}
We give an explicit bound for
{$C$}.
We conjecture that this result is true for any reductive group scheme and when $F$ ranges {(in addition)} over all local fields of characteristic $0$.

Different aspects of this conjecture were studied in \cite{Del,SV,KO,KS}

\end{abstract}

\tableofcontents

\section{Introduction}

Let $G$ be a reductive group over a field $k$. A normal $G$ variety $X$ is called spherical if, for any Borel subgroup $B \subset G$, there is a dense $B$-orbit in $X$. If the characteristic of $k$ is zero, the $G$-representation $O_X(X)$ is multiplicity-free. If $k$ is a  local field of characteristic $0$, one can consider the $G(k)$-representation on the space $C^{\infty}(X(k))$. In this case, it is conjectured that the multiplicity $\dim\Hom (\rho, C^{\infty}(X(k)))$, of every irreducible $G(k)$-representation $\rho$, is finite, and, moreover, that these multiplicities are bounded uniformly in $\rho$. Parts of this conjecture were proved in special cases; see \cite{Del,SV,KO,KS}.

In this paper, we consider an analog question for finite fields:

\begin{introconjecture} \label{conj:intro} Let $G$ be a smooth reductive group scheme over $\mathbb{Z}$ acting on a scheme $X$. Assume that, for every prime $p$, the variety $X_{\mathbb{F}_p}:=X \times \Spec(\mathbb{F}_p)$ is $G_{\mathbb{F}_p}$-spherical. Then there is a constant $N$ such that, for any {finite field $F$} and every irreducible representation $\rho$ of $G(F)$, the multiplicity of $\rho$ in $\mathbb{C}[X(F)]$ is less than $N$.
\end{introconjecture}

We prove:

\begin{introtheorem} \label{thm:intro.main} Conjecture \ref{conj:intro} holds if, in addition, $G$ has type $A$, i.e., if, for every geometric fiber $G_k$ (where $k$ is any algebraically closed field), the universal cover of $[G_k,G_k]$ is a product of special linear groups.
\end{introtheorem}

\begin{remark} We prove a stronger statement (with an explicit bound on the multiplicity) in Theorem \ref{thm:main}.
\end{remark}
\subsection{Motivation}
The conjectures above are motivated by the analysis of principal series representations. The multiplicity of a principal series representation is the dimension of the space of  $B$-equivariant functions on $X$ with respect to corresponding character. This dimension can be bounded in terms of the number of $B$-orbits on $X$.

Experience shows that the multiplicities of  principal series representation bound the multiplicities of all irreducible representations. A possible explanation for this phenomenon is that representations tend to become principal series representations after a base change, and that base change should not affect the multiplicities or at least affect it in  a controllable manner.

\subsection{Idea of the proof}
The proof for $GL_n$ and its forms is based on Lusztig's theory of character sheaves, which is simpler for $GL_n$ than for other groups. In particular, any character sheaf of $GL_n$ is induced from a character sheaf of a torus. Using this, we prove a categorified version of Theorem \ref{thm:intro.main} by categorifying the argument for the principal series. Namely, instead of bounding the multiplicity, we bound the weight, the homological dimension, and the highest homology of a certain $\ell$-adic complex that corresponds to the multiplicity (see Lemma \ref{lem:mult.cx}). This already gives a uniform bound on the limit of the multiplicities when the field is $F=\mathbb F_{p^n}$ and $n$ tends to infinity. Since the multiplicities have a geometric nature {and take integer values}, we show that the sequence of multiplicities is periodic in $n$ (see Lemma \ref{lem:arith.prog}), so the bound holds also for small $n$.

The reduction of general type A case to the case of $GL_n$ is done by classical representation theoretic technique. This technique allows us to control the multiplicity when a finite group 
is replaced by a subgroup of a given index, by a co-central subgroup, or by its extension (see Lemmas \ref{lem:mult.iso}, \ref{lem:mult.rad}). Any group of type A can be related to a form of product of general linear group by such modifications (see Lemma \ref{lem:hat}).

\subsection{Structure of the paper}
In \S\ref{sec:GL} we prove an explicit version of  Theorem \ref{thm:intro.main} for forms of product of general linear groups (see Theorem \ref{thm:GLn}). {In \S\S\ref{ssec:GL.form} we formulate this version.} In \S\S\ref{sec:char.shif} we give an overview of the relevant parts of the theory of character sheaves{; these are summed up in Theorem \ref{thm:CS}}. In \S\S\ref{sec:mult.com} we introduce an $\ell$-adic complex that measures the multiplicity (see Lemma \ref{lem:inp}) and give bounds on this complex (see Lemma \ref{lem:mult.cx}) . In \S\S\ref{sec:GL.pf} we prove the theorem.

In \S\ref{sec:red} we deduce an explicit version of Theorem \ref{thm:intro.main} for all group of type $A$ (see Theorem \ref{thm:main}).

{In Appendix \ref{app:chcs} we explain how to deduce Theorem \ref{thm:CS} from the existing literature.}
{
\subsection{Acknowledgement}
We thank Roman Bezrukavnikov,  Victor Ginzburg, and George Lusztig for explaining us the theory of character sheaves and its relation with irreducible representations of finite groups of Lie type.

We also thank Joseph Bernstein and Eitan Sayag for useful conversations.

{A.A. was partially supported by ISF grant 687/13 
and a Minerva foundation grant. N.A. was partially supported
by NSF grant DMS-1303205. Both of us were partially supported by
BSF grant 2012247.}
}
\section{The $\GL$ case}\label{sec:GL}

{\subsection{Formulation of the main result for the $GL$ case}\label{ssec:GL.form}}
In order to formulate the theorem in the $\GL_n$ case, we use the following

\begin{notation}

 Let $G$ be an algebraic group. Assume $G$ acts on an algebraic variety $Y$.
\begin{itemize}
\item  Let $Y_G=\left\{ (g,y) \in G \times Y \mid g \cdot y=y \right\}$.
\item Assume $G$ is reductive with Borel $B$. Denote the number of absolutely irreducible components of $(Y \times G/B)_G$ by $c(G,Y)$.
\end{itemize}
\end{notation}
By \cite[1.1]{Br} (following \cite{Vi,Kn}), $Y$ is spherical if and only if there are finitely many $B$ orbits in $Y$. This is also equivalent to $\dim (Y \times G/B)_G=\dim G$.

\begin{definition} Let $G$ be an algebraic group over a field $k$. We say that $G$ has type $\GL$ if $G_{\overline{k}}$ is a product of general linear groups.
\end{definition}

The main result of this section {is:}

\begin{theorem} \label{thm:GLn} Let $G$ be an algebraic group of type $\GL$ over a finite field $\mathbb{F}_q$. Let $X$ be a spherical $G$-space. Then, for every irreducible representation $\rho$ of $G(\mathbb{F}_q)$,
\[
\dim\Hom(\rho,\mathbb{C}[X(\mathbb{F}_q)] \leq c(G,X).
\]
\end{theorem}

\subsection{Character sheaves}\label{sec:char.shif}

The proof of Theorem \ref{thm:GLn} is based on Lusztig's theory of character sheaves. We very briefly summarize this theory for $\GL_n$, assuming the reader is familiar with perverse sheaves and Weil sheaves (see \cite{BBD,KW}). Fix a prime power $q$, a prime $\ell$ prime to $q$, and an identification $\mathbb{C} = \overline{\mathbb{Q}_\ell}$. For a complex $\mathcal{F}$ of $\ell$-adic sheaves on an algebraic variety $X$ together with a Weil structure $\alpha$, we denote the corresponding function by $\chi_{\mathcal{F},\alpha}:X(\mathbb{F}_q) \rightarrow \overline{\mathbb{Q}_\ell}\simeq \mathbb{C}$.

\begin{definition} \label{defn:CS} Let $G$ be a reductive group over $\overline{\mathbb{F}_q}$ with Borel $B$ and maximal torus $T \subset B$. Denote the projection $B \rightarrow B/[B,B]=T$ by $\pi_B$.\begin{enumerate}
\item A character sheaf on $T$ is a one-dimensional $\mathbb{Q}_\ell$ local system $\mathcal{L}$ on $T$ such that $\mathcal{L} ^{\otimes n}$ is trivial, for some $n$ prime to $q$.
\item Consider the diagram $G \leftarrow^{\pi} (G/B)_G \leftarrow^{\tau_1} V \rightarrow^{\tau_2} T$, where $V=\left\{ (g,h) \mid ghg ^{-1} \in B\right\}$, $\tau_1(g,h)=(h,B^g)$, $\tau_2(g,h)=\pi_B(ghg ^{-1})$, and $\pi$ is the projection. For a character sheaf $\mathcal{L}$ on $T$, there is a local system $\mathcal{K}$ on $(G/B)_G$ such that $\tau_1^* \mathcal{K}\cong \tau_2^* \mathcal{L}$. Let $\ind_T \mathcal{L} =R\pi_* \mathcal{K}$.
\item Since $\pi$ is semi-small, $\ind_T \mathcal{L}$ is perverse. An induced character sheaf on $G$ is an irreducible direct summand of $\ind_T \mathcal{L}$, for some character sheaf $\mathcal{L}$ on $T$.
\end{enumerate}
\end{definition}

\begin{theorem} \label{thm:CS} Let $G$ be an algebraic group of type $\GL$ defined over $\mathbb{F}_q$. For every irreducible character $\rho$ of $G(\mathbb{F}_q)$, there is an induced character sheaf $\mathcal{M}$ together with a Weil structure $\alpha :\Frob_q ^* \mathcal{M} \rightarrow \mathcal{M}$ which is pure of weight zero, such that $\chi_{M,\alpha}=\rho$.

\end{theorem}

This theorem follows from \cite{Lus_CS,Lus_book,Sho}. We explain this in Appendix \ref{app:chcs}

\subsection{Bounds on the multiplicity complex} \label{sec:mult.com}
{With the notations of Theorem \ref{thm:GLn},} let $\rho_{\mathbb{C}[X(\mathbb{F}_q)]}$ be the character of the module $\mathbb{C}[X(\mathbb{F}_q)]$. {{By Theorem \ref{thm:CS}, i}n order to prove Theorem \ref{thm:GLn},} it is enough to prove that, for each induced character sheaf $\mathcal{M}$ with {pure} Weil structure $\alpha$ {of weight zero},
 we have $\langle \chi_{\mathcal{M},\alpha},\rho_{\mathbb{C}[X(\mathbb{F}_q)]} \rangle \leq c(G,X)$.

Fix such $\mathcal{M},\alpha$. By definition, there is a character sheaf $\mathcal{L}$ on a maximal torus $T$ such that $\mathcal{M}$ is a direct summand of $\ind_T \mathcal{L}$. Choose a Borel $B \supset T$ and consider the diagram
\[
\xymatrix{& (X \times G/B)_G \ar@{->}[dr]^{\widetilde{f}} \ar@{->}[dl]_{\widetilde{\pi}} \ar@{->}[dd]^{p} & & \\ X_G \ar@{->}[dr]_{f} & & (G/B)_G \ar@{->}[dl]^{\pi} & V \ar@{->}[l]_{\ \ \ \ \ \tau_2} \ar@{->}[d]^{\tau_1 } , \\ & G \ar@{->}[d]^{q} & & T \\ & pt & & &}
\]
where $f,\pi,\widetilde{f},\widetilde{\pi},q,p$ are the projections and $\tau_i$ are as in Definition \ref{defn:CS}. Let $\mathcal{K}$ be the local system on $(G/B)_G$ such that $\tau_2^* \mathcal{K} = \tau_1^* \mathcal{L}$, and let $\widetilde{\mathcal{K}}=\widetilde{f}^* \mathcal{K}$. Denote the trivial one-dimensional local system on $X_G$ by $\underline{\mathbb{Q}_\ell}$, and give it the identity Weil structure.

\begin{lemma} \label{lem:inp} {Let $\beta$ be the natural Weil structure on }$Rq_! \left( \mathcal{M} \otimes^L Rf_! \underline{\mathbb{Q}_\ell} \right)$.
 Then, $\chi_{Rq_! \left( \mathcal{M} \otimes^L Rf_! \underline{\mathbb{Q}_\ell} \right), \beta}=|G(\mathbb{F}_q)| \cdot  \langle \chi_{\mathcal{M},\alpha},\rho_{\mathbb{C}[X(\mathbb{F}_q)]} \rangle $.
\end{lemma}

\begin{proof} By the Grothendieck--Lefschetz trace formula, for every $g\in G(\mathbb{F}_q)$,
\[
\chi_{Rf_! \underline{\mathbb{Q}_\ell}}(g)=\left| f ^{-1} (g){(\bF_q)}\right|=\rho_{\mathbb{C}[X(\mathbb{F}_q)]}(g),
\]
and
$$
\chi_{Rq_! \left( \mathcal{M} \otimes^L Rf_! \underline{\mathbb{Q}_\ell} \right)}
=\sum_{g\in G{(\bF_q)}}\chi_{\mathcal{M} \otimes^L Rf_! \underline{\mathbb{Q}_\ell}}(g)=
\sum_{g\in G{(\bF_q)}}\chi_{\mathcal{M}}(g) \cdot \chi_{Rf_! \underline{\mathbb{Q}_\ell}}(g)=
 |G(\mathbb{F}_q)| \cdot  \langle \chi_{\mathcal{M}}, \rho_{\mathbb{C}[X(\mathbb{F}_q)]} \rangle.$$

\end{proof}
{We now study the weights and the cohomologies of $q_! \left( \mathcal{M} \otimes f_! \underline{\mathbb{Q}_\ell} \right)$}
\begin{lemma}\label{lem:K} $\ind_T \mathcal{L} \otimes^L Rf_! \underline{\mathbb{Q}_\ell}= Rp_! \widetilde{\mathcal{K}}$.
\end{lemma}

\begin{proof}
\[
Rp_! \widetilde{\mathcal{K}}=R\pi_! R\widetilde{f}_! \widetilde{f}^* \mathcal{K} =R\pi_! \left( \mathcal{K} \otimes^L R\widetilde{f}_! \underline{\mathbb{Q}_\ell} \right) = R\pi_! \left( \mathcal{K} \otimes^L R\widetilde{f}_! \widetilde{\pi}^* \underline{\mathbb{Q}_\ell} \right) =R\pi_! \left( \mathcal{K} \otimes^L \pi ^*Rf_! \underline{\mathbb{Q}_\ell} \right) =
\]
\[
= R\pi_! \mathcal{K} \otimes^L Rf_! \underline{\mathbb{Q}_\ell}=\ind_T \mathcal{L} \otimes^L Rf_! \underline{\mathbb{Q}_\ell}.
\]
\end{proof}

The following is well-known:

\begin{lemma} \label{lem:vanishing} Let $X$ be a variety of dimension $d$ and let $\mathcal{F}$ be a one-dimensional local system on $X$. Then $H^i_c(X,\mathcal{F})=0$ unless $i\in \left\{ 0,\ldots,2d \right\}$ and $\dim H^{2d}_c(X,\mathcal{F})$ is bounded by the number of irreducible components of $X$.
\end{lemma}
{The last two lemmas give us:}

\begin{corollary} \label{lem:mult.cx} $Rq_! \left( \mathcal{M} \otimes^L Rf_! \underline{\mathbb{Q}_\ell} \right)$ is a mixed complex of weight $\leq 0$, concentrated in degrees $0,\ldots,2\dim G$, and $\dim H^{2\dim G}(Rq_! \left( \mathcal{M} \otimes^L Rf_! \underline{\mathbb{Q}_\ell} \right)) \leq c(X,G)$.
\end{corollary}

\begin{proof}
{It is enough to prove the corollary when $\mathcal{M}$ is replaced by $\ind_T \mathcal{L}$. By lemma  \ref{lem:K}, it is enough to bound the weights and cohomologies of $Rp_! \widetilde{\mathcal{K}}$.}
The claim about the weight follows from the fact that push forward with proper support maps mixed sheaves of weight $\leq 0$ to mixed sheaves of weight $\leq 0$. The rest of the claims follow from Lemma \ref{lem:vanishing}.
\end{proof}

\subsection{Proof of {the main result for the $GL$ case (}Theorem \ref{thm:GLn}{)}} \label{sec:GL.pf}
{In order to complete the proof of Theorem \ref{thm:GLn} we use:}
\begin{lemma} \label{lem:arith.prog} Suppose $\phi \in \C(x_1,\dots,x_m)$ is a rational function. Let $v\in(\C^\times)^m$ such that $\phi$ is regular at $v^n$, for all $n\in \mathbb{Z}_{> 0}$, and the set $\{\phi(v^n)|n\in \mathbb Z_{>0}\}$ is finite. Then the function  $n\mapsto \phi(v^n)$ is periodic on $\Z.$  
\end{lemma}
{For the proof, we use the following lemma, which is an immediate corollary of the Ax-Grothendieck theorem:
\begin{lemma} \label{lem:alg.smi}
Let $G$ be an algebraic group and $H\subset G$ be a Zariski-closed sub-semigroup. Then $H$ is a subgroup.
\end{lemma}

\begin{proof}[Proof of lemma \ref{lem:arith.prog}]
Let $T$ be the Zariski closure of the set $\{v^n|n\in \mathbb Z_{>0}\}$ in $(\C^\times)^m$. By Lemma \ref{lem:alg.smi}, $T$ is an algebraic group. The assumptions imply that $\phi|_T$ takes finitely many values on its domain of definition. Hence, $\phi|_T$ is regular and constant on connected components. Thus, the function $n\mapsto \phi(v^n)$ factors through a homomorphism to the finite group $\pi_0(T)$. This implies the assertion.
\end{proof}
}

\begin{proof}[Proof of Theorem \ref{thm:GLn}] Let $A^\bullet=Rq_!(\mathcal{M} \otimes^L Rf_! \underline{\mathbb{Q}_\ell})$. By Lemma \ref{lem:inp}, we have that, for all $n\in \mathbb{N}$,
\[
f(n):=\frac{\trace(\alpha^n \mid A^\bullet)}{|G(\mathbb{F}_{q^n})|}=\langle \chi_{\mathcal{M},\alpha^n}, \rho_{\mathbb{C}[X(\mathbb{F}_q)]} \rangle.
\]
Since $\chi_{\mathcal{M},\alpha ^n}$ is the character of a representation, {the RHS is an integer}.

{We have to show that $f(1)\leq c(X,G)$.}
Let $\zeta_1,\dots,\zeta_r $ be the eigenvalues of $\alpha$ on $A^\bullet$ (including repetitions). By Lemma  \ref{lem:mult.cx}, we can assume that $$q^{\dim G}=|\zeta_1|=\cdots=|\zeta_t|>|\zeta_{t+1}|\geq\cdots\geq|\zeta_{r}|,$$ for some $t\leq c(X,G)$. Thus, there is a linear functional $Q\in (\C^{r-t})^*$ such that $$\trace(\alpha^n \mid A^\bullet)=\zeta_1^n+\cdots \zeta_t^n+Q(\zeta_{t+1}^n,\dots,\zeta_r^n).$$

  There is also a linear functional  $P\in (\C^m)^*$ and numbers $\xi_1,\dots,\xi_m\in \C^{\times}$ satisfying $|\xi_i|<q^{\dim G}$ such that $$|G(\mathbb{F}_{q^n})|=q^n+P(\xi_1^n,\dots,\xi_m^n).$$ {We obtain

$$f(n)=\frac{\zeta_1^{n}+\cdots \zeta_t^{n}+Q(\zeta_{t+1}^{n},\dots,\zeta_r^{n})}{q^{n}+P(\xi_1^{n},\dots,\xi_m^{n})},$$
so $\limsup_{n\to \infty}f(n)\leq t$. Since the values of $f$ are integers, this implies that $f$ has finitely many values.}
Thus, by Lemma \ref{lem:arith.prog}, $f$ is periodic. {Therefore,
$$f(1)\leq\limsup_{n\to \infty}f(n)\leq t\leq c(X,G). $$}

\end{proof}

\section{Reduction to the $\GL$ case}\label{sec:red}
{\subsection{Fomulation of the main result}}
The reduction of Theorem \ref{thm:intro.main} to Theorem \ref{thm:GLn} is based on the following lemmas:

\begin{lemma} \label{lem:hat} For every reductive group $G$ of type $A$ over $\mathbb{F}_q$ there are groups $G' \subset G''$ and isogeny $i:G' \rightarrow G$, all defined over $\mathbb{F}_q$, such that $G''=G' \cdot \Rad(G'')$ and {$G''$ is of type $GL$.}
\end{lemma}
\begin{proof} Let $G_1:=[G,G]$ be the commutator subgroup of $G$, and let $\widetilde G_1$ be its universal cover.  Set $G'= \widetilde G_1\times Z(G)$ and $i:G'\to G$ be the natural map. Using the clasification of reductive groups over finite fields, we can write  $$\widetilde G_1=\prod_i Res_{E_i/\mathbb{F}_q}(SL_{n_i})\times \prod_j Res_{E_j/\mathbb{F}_q}(SU_{n_j}),$$ for some finite extensions $E_i$ of $\mathbb{F}_q$. Setting $$G''=Z(G)\times\prod_i Res_{E_i/\mathbb{F}_q}(GL_{n_i})\times \prod_j Res_{E_j/\mathbb{F}_q}(U_{n_j}),$$
the result follows.
\end{proof}

\begin{notation} For a reductive group $G$ of type $A$ over $\mathbb{F}_q$, define
\[
d(G)=\min \left\{ |\ker(i)| ^{2}\cdot |\pi_0(G' \cap \Rad(G''))| \mid \text{$G',G'',i$ are as in Lemma \ref{lem:hat}}\right\}
\]
\end{notation}
\begin{notation} Let $\Gamma$ be a finite group acting on a finite set $\Omega$. Denote $$\mu(\Gamma,\Omega):=\max _{\rho\in \irr(\Gamma)}\dim \Hom_{\Gamma} (\rho,\mathbb{C}[\Omega])$$
\end{notation}

The following theorem implies Theorem \ref{thm:intro.main}.

\begin{theorem} \label{thm:main} Let $G$ be a reductive algebraic group of type $A$ over a finite field $\mathbb{F}_q$. Let $X$ be a spherical $G$-space. Then,
\[
\mu(G(\mathbb{F}_q),X(\mathbb{F}_q) )\leq d(G)c(G,X).
\]
\end{theorem}

{\subsection{Proof of the main result (Theorem \ref{thm:main})}}
For the proof, we will need the following:

\begin{lemma}\label{lem:mult.iso}
Let $i:\Gamma'\to \Gamma$ be a morphism of finite groups and $\Omega$ be a $\Gamma$-space. Then $$\mu(\Gamma,\Omega) \leq [\Gamma : i(\Gamma ')] \mu(\Gamma ',\Omega).$$
\end{lemma}

\begin{proof} Let $\rho$ be an irreducible representation of $\Gamma$ and let $\Res_{i(\Gamma ')}^{\Gamma} \rho=\tau_1 \oplus \ldots \oplus \tau_k$ a decomposition into irreducible components. Since
\[
k \leq \dim \End_{i(\Gamma ')}(\Res_{i(\Gamma)}^{\Gamma} \rho)=\dim \Hom_{\Gamma}(\rho,\Ind_{i(\Gamma)}^{\Gamma '} \Res_{i(\Gamma)}^{\Gamma '}  \rho) \leq \frac{\dim \Ind_{i(\Gamma ')}^{\Gamma} \Res_{i(\Gamma ')}^{\Gamma} \rho}{\dim \rho}=[\Gamma : i(\Gamma ')],
\]
we have
\[
\dim \Hom_{\Gamma}(\rho, \mathbb{C}[\Omega]) \leq \sum_i \dim \Hom_{i(\Gamma ')}(\tau_i, \mathbb{C}[\Omega]) \leq [\Gamma : i(\Gamma ')] \mu(i(\Gamma'),\Omega) \leq [\Gamma : i(\Gamma ')] \mu(\Gamma',\Omega).
\]
\end{proof}

\begin{lemma}\label{lem:mult.rad} Suppose that $ \Gamma \subset \Gamma'$ are finite groups and $\Omega$ is a $\Gamma$-set. Let $\Omega\times_{\Gamma} \Gamma'$ be the quotient of $\Omega\times \Gamma'$ by the diagonal action of $\Gamma$. Then,
\[
\mu (\Gamma,\Omega) \leq [\Gamma':\Gamma \cdot Z(\Gamma')] \cdot \mu (\Gamma',\Omega\times_{\Gamma} \Gamma')
\]
\end{lemma}
\begin{proof}
Any representation of $\Gamma$ can be extended to $\Gamma \cdot Z(\Gamma')$.
Since $\mathbb{C}[\Omega \times_\Gamma (\Gamma \cdot Z(\Gamma '))]=\Ind_\Gamma ^{\Gamma\cdot Z(\Gamma ')} \mathbb{C}[\Omega]$, we have $$\mu (\Gamma,\Omega) \leq \mu (\Gamma\cdot  Z(\Gamma'),\Omega\times_{\Gamma}( \Gamma\cdot  Z(\Gamma')))\leq\mu (\Gamma\cdot  Z(\Gamma'),\Omega\times_{\Gamma} \Gamma').$$ Let $\rho$ be an irreducible representation of $\Gamma \cdot Z(\Gamma ')$. We have

\begin{align*}
\dim\Hom_{\Gamma \cdot Z(\Gamma ')}(\rho,\C[\Omega\times_{\Gamma} \Gamma']) &=\dim\Hom_{\Gamma'}(\ind_{\Gamma\cdot  Z(\Gamma')}^{\Gamma'}(\rho),\C[\Omega\times_{\Gamma} \Gamma'])\\
&\leq \mathrm length(\ind_{\Gamma\cdot  Z(\Gamma')}^{\Gamma'}(\rho)) \cdot\mu (\Gamma',\Omega\times_{\Gamma} \Gamma')\\
&\leq \dim\End_{\Gamma'}(\ind_{\Gamma\cdot  Z(\Gamma')}^{\Gamma'}(\rho)) \cdot\mu (\Gamma',\Omega\times_{\Gamma} \Gamma')\\
&=\dim\Hom_{\Gamma\cdot  Z(\Gamma')}(\rho,\ind_{\Gamma\cdot  Z(\Gamma')}^{\Gamma'}(\rho)) \cdot\mu (\Gamma',\Omega\times_{\Gamma} \Gamma')\\
&\leq \frac{\dim\ind_{\Gamma\cdot  Z(\Gamma')}^{\Gamma'}(\rho)}{\dim \rho} \cdot\mu (\Gamma',\Omega\times_{\Gamma} \Gamma')=[\Gamma':\Gamma \cdot Z(\Gamma')] \cdot \mu (\Gamma',\Omega\times_{\Gamma} \Gamma').
\end{align*}

\end{proof}

\begin{lemma}\label{lem:c.rad} Suppose that $ G \subset G'$ are connected, reductive, and quasi-split algebraic groups such that $G'=G \cdot \Rad(G')$, and let $X$ be a $G$-variety. Then $$c(G,X) \geq c(G',X \times_G G').$$
\end{lemma}

\begin{proof} Denote $X'=X \times_G G'$. The flag varieties of $G$ and $G'$ are isomorphic; we denote them by $\mathcal{F}$. The map
$$\xymatrix{
 \Rad(G') \times ((X \times \mathcal{F}) \times G)\ar@{->}[rrrr]_{\ \ \ (r,((x,f),g)) \mapsto (((x,r),f),g)}   &&&& ((X \times G' )\times \mathcal{F}) \times {G'}
 }$$
induces an onto map $\Rad(G') \times (X\times \mathcal{F})_G \rightarrow (X'\times \mathcal{F})_{G'}$.

Thus the number of components of $\Rad(G') \times (X\times \mathcal{F})_G$ is not smaller than the number of components of $(X'\times \mathcal{F})_{G'}$,  which implies the assertion.
%
\end{proof}

\begin{lemma}\label{lem:c.iso}
Let $i:G'\to G$ be an isogeny of connected reductive algebraic groups and let $X$ be a $G$ space. Then $$c(G',X)\leq |ker(i)| c(G,X)$$
\end{lemma}

\begin{proof} Let $\mathcal{F}$ be the flag variety of $G$ and $G'$. The map $\pi : X \times \mathcal{F} \times G' \rightarrow X \times \mathcal{F} \times G$ defined by $(x,f,g) \mapsto (x,f,i(g))$ is an etale map whose fibers have sizes $|\ker(i)|$. Therefore, the number of irreducible components of $\pi ^{-1} (X \times_G \mathcal{F})=X\times_{G'} \mathcal{F}$ is at most $|\ker(i)|c(G,X)$.
\end{proof}

\begin{proof}[Proof of Theorem \ref{thm:main}] If $\varphi :L \rightarrow L'$ is an epimorphism of algebraic groups over $\mathbb{F}_q$ {such that $\ker \varphi$ is abelian}, then $$[L'(\mathbb{F}_q):\varphi(L(\mathbb{F}_q))] \leq |H^1(\Gal_{\mathbb{F}_q},\ker \varphi)| = |H^1(\Gal_{\mathbb{F}_q},\pi_0(\ker \varphi))| \leq | \pi_0(\ker \varphi) |.$$ Therefore,

\begin{eqnarray*}
\mu(G(\mathbb{F}_q),X(\mathbb{F}_q) ) & \stackrel{\text{Lemma \ref{lem:mult.iso}}}{\leq}&  [G(\mathbb{F}_q):i(G'(\mathbb{F}_q))]\mu(G'(\mathbb{F}_q),X(\mathbb{F}_q)) \\
& \stackrel{}{\leq} &|\ker(i)|\mu(G'(\mathbb{F}_q),X(\mathbb{F}_q) )\\
&\stackrel{\text{Lemma \ref{lem:mult.rad}}}{\leq}&|\ker(i)|[G''(\mathbb F_q):G'(\mathbb F_q)Z(G''(\mathbb F_q))]\mu(G''(\mathbb{F}_q),X(\mathbb{F}_q)\times_{G'(\mathbb{F}_q)}{G''(\mathbb{F}_q)}])\\
&{\leq} &|\ker(i)|[G''(\mathbb F_q):G'(\mathbb F_q)Z(G''(\mathbb F_q))]\mu(G''(\mathbb{F}_q),(X\times_{G'}G'')(\mathbb{F}_q)\\
&\stackrel{}{\leq}& |\ker(i)||\pi_0(G' \cap \Rad(G''))|\mu(G''(\mathbb{F}_q),(X\times_{G'}G'')(\mathbb{F}_q))\\
&\stackrel{\text{Theorem \ref{thm:GLn}}}{\leq}&|\ker(i)||\pi_0(G' \cap \Rad(G''))|c(G'',X\times_{G'}G'')\\
&\stackrel{\text{Lemma \ref{lem:c.rad}}}{\leq}&|\ker(i)||\pi_0(G' \cap \Rad(G''))|c(G',X)\\
&\stackrel{\text{Lemma \ref{lem:c.iso}}}{\leq}&|\ker(i)|^{2}|\pi_0(G' \cap \Rad(G''))|c(G,X)= d(G)c(G,X).
\end{eqnarray*}
\end{proof}

\appendix \section{Characters and character sheaves} \label{app:chcs}

The goal of this appendix is to explain how to derive Theorem \ref{thm:CS} from the literature.

In \S\S\ref{subsec:ch.induced} we discuss the general notion of a character sheaf and explain that, for groups of type $\GL$, they are all induced (i.e. coincide with Definition \ref{defn:CS}).

In  \S\S\ref{subsec:chcs} we discuss the classification of characters and their relation with character sheaves, and explain that characters of groups of type $\GL$ can be obtained by the sheaf to function correspondence from character sheaves.

\subsection{All character sheaves are induced for $\GL_n$} \label{subsec:ch.induced}

In this subsection, all reductive groups are defined over the algebraic closure of a finite field $\mathbb{F}_q$.

Definition \ref{defn:CS} is a special case of the definitions \cite[I 2.10, 4.1]{Lus_CS} of the notion of a character sheaf on a reductive group and the notion of induction of character sheaves from a Levi subgroup to a reductive group. In addition, Lusztig defines in \cite[I 3.10]{Lus_CS} the notion of cuspidal character sheaves, and shows in \cite[I 4.4]{Lus_CS} that all irreducible character sheaves are direct summands of character sheaves induced from irreducible cuspidal character sheaves.

\begin{proposition} Let $G=\prod_i \GL_{n_i}$ and assume that $n_1>1$. Then there are no irreducible cuspidal character sheaves on $G$.
\end{proposition}

\begin{proof} {Assume the contrary, and let $\mathcal{A}$ be an irreducible cuspidal character sheaf on $G$.} By \cite[II.7.1.1]{Lus_CS}, there is an integer $k$ prime to $q$ such that $\mathcal{A}$ is equivariant with respect to the action of $G \times Z(G)$ on $G$ given by $(g,z) \cdot h=g ^{-1} h g z^k$. Let $O$ be the orbit of a regular unipotent element. By \cite[II.7.1.2]{Lus_CS} $\mathcal{A}$ is the IC-extension of a $G \times Z(G)$-equivariant local system $\mathcal{E}$ on $O$. Since the centralizer of every element in $O$ is connected, the restriction of $\mathcal{E}$ to any $G$-orbit is trivial. By \cite[II.7.1.3]{Lus_CS}, for every parabolic $P=LU \neq G$ and any $g\in L$, we have
\begin{equation}\label{eq:zero}
H_c^{\dim\, O / Z(G) -\dim \Ad(G)g}(gU \cap O, \mathcal{E}|_{gU \cap O})=0.
\end{equation}

Taking $P=B$ and $g=1$, the variety $U \cap O$ is open in $U$, is contained in a $G$-orbit, and $\dim\, O / Z(G) -\dim \Ad(G)g=2\dim U$. Hence,
\begin{equation}
H_c^{\dim\, O / Z(G) -\dim \Ad(G)g}(gU \cap O, \mathcal{E}|_{gU \cap O})=H_c^{2\dim U}(U\cap O, \underline{\mathbb{Q}_\ell}) \neq 0,
\end{equation}
{contradicting \eqref{eq:zero}}.
\end{proof}

\begin{corollary} \label{cor:all.induced} Let $G=\prod_i \GL_{n_i}$. Then all character sheaves on {$G$} are induced (in the sense of Definition \ref{defn:CS}).
\end{corollary}

\subsection{Characters and character sheaves} \label{subsec:chcs}

Let $G$ be a connected reductive group defined over $\mathbb{F}_q$. We summarize the classification of irreducible characters of $G(\mathbb{F}_q)$ described in \cite{Lus_book}:
\begin{enumerate}
\item \label{itm:ELn} For every integer $n$, which is prime to $q$, and every line bundle $L$ over the flag variety of $G$, Lusztig defines a set $\mathcal{E}_{L,n} \subset \Irr G(\mathbb{F}_q)$
(\cite[2.19]{Lus_book}), and shows in \cite[6.5]{Lus_book} that every irreducible representation of $G(\mathbb{F}_q)$ is contained in one of the sets $\mathcal{E}_{L,n}$.
\item \label{itm:WLngamma} For every $L$ and $n$ as above, Lusztig attaches in \cite[1.8]{Lus_book} a Coxeter group $W_{L,n}$, together with an automorphism $\gamma$ of $W_{L,n}$ (\cite[2.15.1]{Lus_book}). In the case $G$ has type $\GL$, the Coxeter group $W_{L,n}$ is a product of symmetric groups.
\item For every Coxeter group $W$, Lusztig defines the notion of a family (\cite[4.2]{Lus_book}), which is a subset of $\Irr_{\mathbb{Q}} W$ (the collection of isomorphism classes of irreducible $\mathbb{Q}$-representations of $W$). The set of families forms a partition of $\Irr_\mathbb{Q} W$. In the case $W$ is a product of symmetric groups, families are singletons (\cite[4.4]{Lus_book}).
\item For every family $\mathcal{F}$, Lusztig attaches a finite group $\mathcal{G}_\mathcal{F}$ (\cite[4.4--4.13,4.21]{Lus_book}). In the case $W$ is a product of symmetric groups, $\mathcal{G}_\mathcal{F}$ is the trivial group, for every family $\mathcal{F}$ (\cite[4.4]{Lus_book}).
\item For every family $\mathcal{F}$ which is $\gamma$-stable\footnote{$\gamma$ is assumed to come form an automorphism of the root system of $W$, which is the case for the $\gamma$ constructed in \ref{itm:WLngamma}}, Lusztig defines in \cite[4.18--4.21]{Lus_book} a group $\widetilde{\mathcal{G}}_\mathcal{F}$ which is an extension of $\mathcal{G}_\mathcal{F}$ by $\mathbb{Z}$.
\item For an extension $1 \rightarrow \mathcal{G}\rightarrow \widetilde{\mathcal{G}} \rightarrow \mathbb{Z}\rightarrow 1$, Lusztig defines in \cite[4.16,4.21]{Lus_book} two sets, $\mathfrak{M}(\mathcal{G} \subset \widetilde{\mathcal{G}})$, and $\overline{\mathfrak{M}}(\mathcal{G} \subset \widetilde{\mathcal{G}})$. In the case $\mathcal{G}$ is trivial, $\mathfrak{M}(\mathcal{G} \subset \widetilde{\mathcal{G}})$ is the set of roots of unity, and $\overline{\mathfrak{M}}(\mathcal{G} \subset \widetilde{\mathcal{G}})$ is a singleton.
\item \label{itm:bijection} In \cite[4.23]{Lus_book}, Lusztig constructs a bijection between $\mathcal{E}_{L,n}$ and $$\overline{X}(W_{L,n},\gamma):=\bigsqcup \overline{\mathfrak{M}}(\mathcal{G}_\mathcal{F} \subset \widetilde{\mathcal{G}}_\mathcal{F}),$$ where the union is over $\gamma$-stable families in $\Irr_\mathbb{Q} W$.
\end{enumerate}

Based on this classification Lustig defined the notion of an almost character in the following way:
\begin{enumerate} \setcounter{enumi}{7}
\item For $\mathcal{G},\widetilde{\mathcal{G}}$ as above, Lusztig defines in \cite[4.21]{Lus_book} a map $$\left\{ \, ,\, \right\} : \mathfrak{M}(\mathcal{G} \subset \widetilde{\mathcal{G}}) \times \overline{\mathfrak{M}}(\mathcal{G} \subset \widetilde{\mathcal{G}}) \rightarrow \overline{\mathbb{Q}_\ell}.$$ In the case $\mathcal{G}$ is trivial, the pairing is the projection to the second coordinate.
\item \label{itm:almost.char} For every family $\mathcal{F}$ and every $x\in \mathfrak{M}(\mathcal{G}_\mathcal{F} \subset \widetilde{\mathcal{G}}_\mathcal{F})$, Lusztig defines in \cite[4.24.1]{Lus_book} a virtual representation $R_{\mathcal{F},x}$ of $G(\mathbb{F}_q)$. In view of the above, this definition implies that, in the case where $G$ has type $\GL$, $R_{x,\mathcal{F}}$ is a root of unity times the irreducible representation corresponding to the single element of $\overline{\mathfrak{M}}(\mathcal{G}_\mathcal{F} \subset \widetilde{\mathcal{G}}_\mathcal{F})$ under the bijection in \ref{itm:bijection}. These elements are called almost characters, see \cite[13.6]{Lus_book}.
\end{enumerate}

The relation between almost characters and character sheaves is given by the following:
\begin{enumerate} \setcounter{enumi}{9}
\item \label{itm:shoji} \cite[Theorem 5.7]{Sho} implies that, if $G$ is a group of type $\GL$ and $R$ is an almost character of $G$, then there is a Frobenius-invariant character sheaf $A$, a Weil structure on $A$ which is pure of weight $\dim\, G-\dim\supp A$ (\cite[1.4.1]{Sho}), and a root of unity $\zeta$ such that $R=\zeta \chi_A$.
\end{enumerate}

Theorem \ref{thm:CS} follows now from \ref{itm:ELn},\ref{itm:bijection},\ref{itm:almost.char},\ref{itm:shoji}, and Corollary \ref{cor:all.induced}.

\end{document}